\documentclass[a4paper,12pt,twoside]{article}
\usepackage{amsmath,amsthm,amscd,amsfonts,amssymb,amsxtra}
\usepackage[hmargin=1.2in,vmargin=1.2in]{geometry}
\usepackage{graphicx}
\usepackage{array,longtable}
\usepackage[pdftex,bookmarks,colorlinks=false]{hyperref}
\usepackage{booktabs,enumitem,caption,subcaption}
\usepackage[mathscr]{euscript}

\usepackage[auth-sc-lg,affil-it]{authblk}
\numberwithin{equation}{section}

\allowdisplaybreaks[4]

\newtheorem{thm}{Theorem}[section]
\newtheorem{defn}{Definition}[section]
\newtheorem{lem}[thm]{Lemma}
\newtheorem{prop}[thm]{Proposition}
\newtheorem{cor}[thm]{Corollary}

\newtheorem*{rem}{Remark}
\newtheorem{prob}{Problem}
\newtheorem{app}{Application}[subsection]

\def\ni{\noindent}
\def\N{\mathbb{N}}
\def\J{\mathbb{J}}
\def\cB{\mathcal{B}}

\title{\textbf{\sc Jaco-Type Graphs and Black Energy Dissipation}}

\author{Johan Kok}
\affil{\small Tshwane Metropolitan Police Department\\ City of Tshwane, South Africa \\ {\tt kokkiek2@tshwane.gov.za.}}

\author{N. K. Sudev}
\affil{\small Centre for Studies in Discrete Mathematics\\ Vidya Academy of Science \&   Technology \\ Thalakkottukara, Thrissur, India.\\ {\tt sudevnk@gmail.com}}

\author{K. P. Chithra}
\affil{\small Naduvath Mana, Nandikkara \\ Thrissur, India.\\ {\tt chithrasudev@gmail.com}}

\author{U. Mary}
\affil{\small Department of Mathematics\\ Nirmala College For Women\\  Coimbatore, India. \\{\tt marycbe@gmail.com}}

\date{}

\begin{document}
\maketitle
 
\vspace{-0.5cm}

\begin{abstract}
\ni In this paper, we introduce the notion of an energy graph as a simple, directed and vertex labeled graph such that the arcs $(u_i, u_j) \notin A(G)$ if $i > j$ for all distinct pairs $u_i,u_j$ and at least one vertex $u_k$ exists such that $d^-(u_k)=0$.  Initially, equal amount of potential energy is allocated to certain vertices. Then, at a point of time these vertices transform the potential energy into kinetic energy and initiate transmission to head vertices. Upon reaching a head vertex, perfect elastic collisions with atomic particles take place and propagate energy further. Propagation rules apply which result in energy dissipation. This dissipated energy is called black energy. The notion of the black arc number of a graph is also introduced in this paper. Mainly Jaco-type graphs are considered for the application of the new concepts.
\end{abstract}
\ni \textbf{Keywords:} Energy graph, black energy, black arc number, black cloud, solid subgraph, Jaco-type graph.

\vspace{0.20cm}

\ni \textbf{Mathematics Subject Classification:} 05C07, 05C38, 05C75, 05C85.

\section{Introduction}

For general notations and concepts in graphs and digraphs see \cite{BM1,CL1,FH,DBW}. Unless mentioned otherwise all graphs in this paper are simple, connected and directed graphs (digraphs). 

\vspace{0.2cm}

We begin by observing that any unlabeled, undirected simple and connected graph of order $n\ge 2$ allows a vertex labeling and orientation such that all arcs $(u_i,u_j) \notin A(G)$ if $i>j$. This can be established as explained below.

\vspace{0.2cm}

Any arbitrary vertex can be labeled $u_1$ and all edges incident with $u_1$ can be orientated as out-arcs. Thereafter, all neighbours of $u_1$ can randomly be labeled  $u_2,u_3,\ldots,u_{d_G(u_1)}$. Sequentially, the resultant edges (if any) of vertex $u_2$ can be orientated as out-arcs to its resultant neighbours, followed by similar orientation of the resultant edges of $u_3,u_4,\ldots,u_{d_G(u_1)}$. Proceeding iteratively, the graph $G$ allows a vertex labeling and orientation as prescribed. 

\vspace{0.2cm}

The above mentioned process of labeling the vertices and assigning orientation to the edges of a given graph is indeed well-known and called {\em topological ordering.} For the purposes of this study the digraph thus obtained is called an \textit{energy graph}. 

\vspace{0.2cm}

The vertex $u_1$ with $d^-(u_1)=0$ is called a \textit{source vertex} and the vertex $u_n$ and possibly others with out-degree equal to zero are called \textit{sink vertices}. 

\vspace{0.2cm}

In view of the above definition on energy graphs, the following is an immediate result.

\begin{lem}\label{Lem-2.1}
An energy graph of order $n\ge 2$ has at least one source vertex and at least one sink vertex.
\end{lem}
\begin{proof}
\textit{Part (i):} Firstly, that some energy graphs have exactly one source vertex follows from the topological ordering. Now, if topological ordering begins with a number of non-adjacent vertices say, $\ell < n$ such vertices, the vertices can randomly be labeled $u_1,u_2,u_3,\ldots,u_\ell$. Completing the vertex labeling and edge orientation as described clearly results in $\ell$ source vertices. Therefore, an energy graph has at least one source vertex.

\textit{Part (ii):}  In a similar way we can establish the existence of at least one sink vertex.
\end{proof}

In general, a simple and connected digraph of order $n\ge 2$ does not necessarily allow a vertex labeling to render it an energy graph. However, it can be observed that a simple, connected acyclic digraph always allows such vertex labeling. 

\subsection{Black Energy Dissipation within an Energy Graph}

The energy propagation model obeys the following propagation rules. An energy graph \textit{per se} has zero mass. Each vertex $u_i$ is allocated a number equal to $d^+(u_i)$ of atomic particles, all of equal mass $m>0$. All the atomic particles are initially at rest at their respective vertices. 

\vspace{0.2cm}

All source vertices $u_k$ are initially allocated an equal amount of potential energy of $\xi >0$ joules. The allocated potential energy is shared equally amongst the atomic particles. At time $t=0$, all source vertices simultaneously ignite kinetic energy, and only one atomic particle per out-arc transits an out-arc towards the head vertex. Hence, when transition ignites for each source vertex $u_k$ with $d^+(u_k)= \ell$ we have $\ell \cdot\frac{1}{2}mv^2 = \xi$ joule. 

\vspace{0.2cm}

Regardless of the true speed $v$ of an atomic particle, the time to transit along an arc is considered to be 1 time unit. Therefore, the minimum travelling time $t$ for some atomic particle to reach a head vertex $u_j$ from $u_i$ is equal to the directed distance between $u_i$ and $u_j$ i.e. $t= d^{\rightarrow}_G(u_i,u_j)$. 

\vspace{0.2cm}

Upon one or more say, $\ell'$  of the first atomic particles reaching a head vertex $u_j$ simultaneously, the atomic particles merge into a single atomic particle of mass $\ell'm$ and engage in a perfect elastic collision with the atomic particles allocated to the head vertex $u_j$. Clearly these atomic particles transitted from one or more source vertices such for any such source vertex say, $u_i$ the directed distance $d^\rightarrow_G(u_i,u_j)$ is a minimum over all directed distances from source vertices to $u_j$. For the perfect elastic collision the laws of conservation of momentum and total energy apply. The kinetic energy of all atomic particles reaching the head vertex $u_j$ later ($d^\rightarrow_G(u_i,u_j)$ not a minimum) than those which arrived first, together with the \textit{mass-energy equivalence,} $mc^2$, ($c$ is the speed of light) dissipate into the surrounding universe. This dissipated energy is called \textit{black energy}, denoted $\mathfrak{E}_G$. The arcs along which these late-coming particles transmitted are called \textit{black arcs}. Hence, a black arc is an in-arc of some vertex $u_j$ that does not lie on a minimum directed distance path from any source vertex $u_i$.

\vspace{0.2cm}

Kinetic energy reaching a sink vertex $u_t$ is stored (capacitated) as potential energy $\xi_{u_t}'$ joule. By the conservation of momentum and total energy it follows that $\mathfrak{E}_G = \sum\limits_{i=1}^{n}mc^2\cdot d^+(u_i) +\sum\limits_{(u_k~a~source~vertex)}\xi_{u_k} - \sum\limits_{(u_t~a~sink~vertex)}\xi'_{u_t} -$ (\textit{total dissipated mass-energy equivalence}). To find closure for the expression, the last term must be determined for an energy graph $G$.

\section{Black Energy Dissipation within Jaco-Type Graphs}

For ease of introductory analysis we consider graphs with well-defined vertex labeling and well-defined orientation. The families of graphs which naturally offers this research avenue are Jaco graphs and Jaco-type graphs. 

\vspace{0.2cm}

The concept of linear Jaco graphs was introduced in \cite{KF1} and studied initially in \cite{KF1,KF2}. Further studies on these graph classes have been reported in \cite{KSS1,JK,KSC1} and following these studies some other significant papers have been published. 

\vspace{0.2cm}

In \cite{KSC1} it is reported that a linear Jaco graph $J_n(x)$ can be defined a the graphical embodiment of a specific sequence. The introductory research (see \cite{KSC1}) dealt with non-negative, non-decreasing integer sequences. This observation has opened a wide scope for the graphical embodiment of countless other interesting integer sequences.  These graphs are broadly termed as \textit{Jaco-type graphs} and the notions of  finite and infinite Jaco-type graphs are as given below.  

\begin{defn}\label{Defn-1.1}{\rm 
\cite{KSC1} For a non-negative, non-decreasing integer sequence $\{a_n\}$, the \textit{infinite Jaco-type graph}, denoted by $J_\infty(\{a_n\})$, is a directed graph with vertex set $V(J_\infty(\{a_n\}))= \{u_i: i \in \N\}$ and the arc set $A(J_\infty(\{a_n\})) \subseteq \{(u_i, u_j): i, j \in \N, i< j\}$ such that $(u_i,u_ j) \in A(J_\infty(\{a_n\}))$ if and only if $i+ a_i \ge j$.
}\end{defn}

\begin{defn}\label{Defn-1.2}{\rm 
\cite{KSC1} For a non-negative, non-decreasing integer sequence $\{a_n\}$, the \textit{finite Jaco-type Graph} denoted by $J_n(\{a_n\})$, is the set of finite subgraphs of the infinite Jaco-type graph $J_\infty(\{a_n\});n\in \N$. 
}\end{defn}

Jaco graphs and Jaco-type graphs can be used for modelling many theoretical and practical problems. In this paper, we discuss some of the applications of these types of graph classes.

\begin{lem}
Finite Jaco-type graphs are energy graphs.
\end{lem}
\begin{proof}
In view of Definition \ref{Defn-1.1} and \ref{Defn-1.2}, it can be noted directly that a finite Jaco-type graph has both a unique source vertex $u_1$ and and primary sink vertex $u_n$.  Furthermore, $(u_i,u_j) \notin A(J_n(\{a_n\}))$ if $i>j$. Hence, every finite Jaco-type graph is an energy graph.
\end{proof}

\subsection{Jaco-Type Graph for the Sequence of Natural Numbers} 

For purpose of notation all sequences will be labeled $s_i$, $i \in \N$. The definition of the infinite Jaco-Type graph corresponding to the captioned sequence can be derived from Definition \ref{Defn-1.1}. We have the graph  $J_\infty(s_1)$, defined by $V(J_\infty(s_1)) = \{u_i: i \in \N\}$, $A(J_\infty(s_1)) \subseteq \{(u_i, u_j): i, j \in \N, i< j\}$ and $(u_i,u_ j) \in A(J_\infty(s_1))$ if and only if $2i\ge j$. Note that a finite Jaco-Type graph $J_n(s_1)$ in this family is obtained from $J_\infty(s_1)$ by lobbing off all vertices $u_k$ (with incident arcs) $\forall\, k > n$. 

\vspace{0.25cm}

The arrival times of the atomic particles at a head vertex $u_j$ will be stringed and denoted $u_j \sim\langle t_1,t_2,t_3,\ldots, t_{d^-_G(u_j)}\rangle$ such that $t_i \le t_l$ for $1\le i,l\le {d^-_G(u_j)}$. For the Jaco-type graph $J_8(s_1)$ we find (see Figure 1): $u_1\sim \langle 0\rangle$, $u_2\sim \langle 1\rangle$, $u_3 \sim \langle 2\rangle$, $u_4\sim \langle 2,3\rangle$ hence, ($\frac{1}{6}\xi + mc^2$) joules transiting along arc $(u_3,u_4)$ dissipate as black energy. Then we have $u_5\sim \langle 3,3\rangle$ so two atomic particles merge to collide perfectly at $u_5$ with $\frac{7}{24}\xi$ joules. Then $u_6\sim \langle3,3,4\rangle$ hence, $\frac{7}{24}\xi$ joules collide perfectly at $u_6$  and ($\frac{7}{72}\xi +mc^2$) joules dissipate as black energy. At vertices $u_7$ and $u_8$ the arrival times are stringed as $u_7\sim \langle 3,4,4\rangle$ and $u_8\sim \langle 3,4,4,4\rangle$ respectively. At $u_7$ and $u_8$ energy amounting to $(\frac{7}{72}\xi + \frac{7}{48}\xi + 2mc^2)$ joules and $(\frac{7}{72}\xi + \frac{7}{48}\xi  + \frac{1}{8}\xi+ 3mc^2)$ joules dissipate as black energy, respectively. Only $(\frac{1}{8}\xi +mc^2)$ joules reach the sink vertex, $u_8$. The energy graph has a total of $(\frac{1}{8}\xi + 9mc^2)$ joules potential energy and mass-energy equivalence, capacitated within the graph. Hence, a total of $(\frac{7}{8}\xi +7mc^2)$ joules dissipated as black energy. Figure \ref{fig:Fig-1} depicts $J_{8}(s_1)$.

\begin{figure}[h!]
\centering
\includegraphics[width=0.6\linewidth]{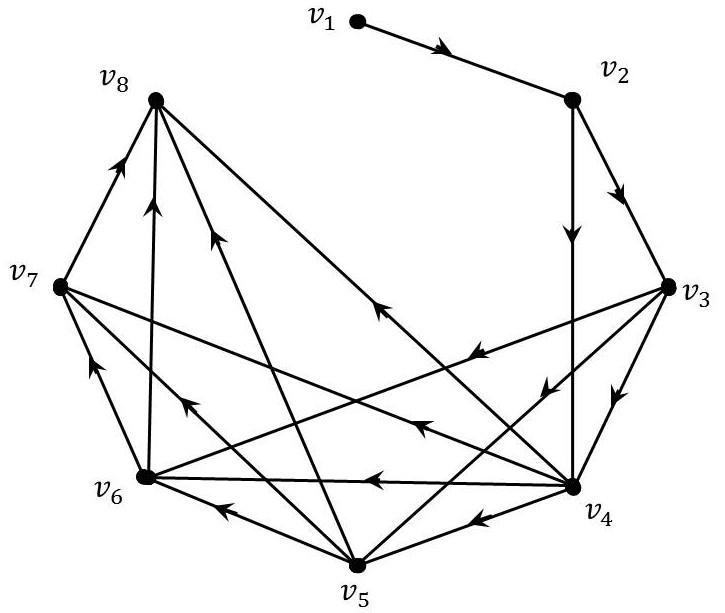}
\caption{$J_{8}(s_1)$.}
\label{fig:Fig-1}
\end{figure}

\subsection{Jaco-Type Graph for the Fibonacci Sequence} 

The definition of the infinite Jaco-Type graph corresponding to the Fibonacci sequence, $s_2=\{f_n\}$, $f_0=0,f_1 = 1,f_2=1$, $f_n=f_{n-1}+f_{n-2}$, $n = 1,2,3,\ldots$ can be derived from Definition \ref{Defn-1.1}. We have the graph  $J_\infty(s_2)$, defined by $V(J_\infty(s_2)) = \{u_i: i \in \N\}$, $A(J_\infty(s_2)) \subseteq \{(u_i, u_j): i, j \in \N, i< j\}$ and $(u_i,u_ j) \in A(J_\infty(s_2))$ if and only if $i + f_i\ge j$. 

\vspace{0.25cm}

For the Jaco-type graph $J_{12}(s_2)$ we find (see Figure \ref{fig:Fig-2}) $v_1\sim \langle 0\rangle$, $u_2\sim \langle 1\rangle$, $u_3 \sim \langle 2\rangle$, $u_4\sim \langle3\rangle$ and $u_5\sim \langle3,4\rangle,u_6\sim \langle 4,4\rangle,u_7\sim \langle 4,4,5\rangle, u_8\sim \langle 4,5,5\rangle, u_9\sim \langle 4,5,5,5\rangle, u_{10}\sim \langle 4,5,5,5,5\rangle, u_{11}\sim \langle 5,5,5,5,5\rangle$ and $u_{12}\sim \langle 5,5,5,5,5,6\rangle$. It implies that only $(\frac{163}{450}\xi + 5mc^2)$ joules kinetic energy capacitate at the sink vertex $u_{12}$. The energy graph has a total of $(\frac{163}{450}\xi + 21mc^2)$ joules potential energy and mass-energy equivalence, allocated within the graph. This means that $(\frac{287}{450}\xi + 12mc^2)$ joules dissipate into black energy. Figure \ref{fig:Fig-2} depicts $J_{12}(s_2)$. 

\begin{figure}[h!]
\centering
\includegraphics[width=0.6\linewidth]{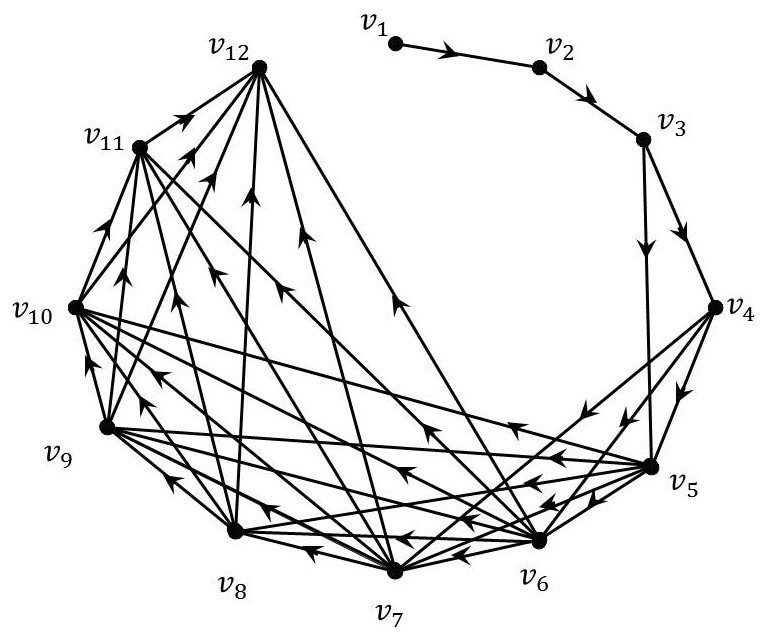}
\caption{$J_{12}(s_2)$.}
\label{fig:Fig-2}
\end{figure}

\subsection{Number of Black Arcs of Certain Graphs}

For a simple connected graph $G$ in general, the number of black arcs dependent on the orientation of a corresponding energy graph. In general an energy graph for a given graph $G$ is not unique. Denote an orientation of a graph by $\varphi(G)$. Denote the number of black arcs in respect of a given $\varphi(G)$ to be $b^{\varphi(G)}(G)$. Define the black arc number, $b^\bullet(G) = \min\{b^{\varphi(G)}(G): \forall\,\ \varphi(G)\}$. First we present a perhaps obvious, but useful lemma.

\begin{lem}\label{Lem-3.2}
If the stringed arrival times of atomic particles at vertex $u_j$ in an energy graph $G$ is $u_j\sim \langle t_1,t_2,t_3,\ldots, t_\ell, t_{\ell+1}, t_{\ell+2},\ldots, t_{d^-_G(u_j)}\rangle$ and $t_1=t_2=t_3=\cdots =t_\ell <t_{\ell+1}\le t_{\ell +2} \le \cdots \le v_{d^-_G(u_j)}$, then $d^-_G(u_j) - \ell$ in-arcs of $u_j$ dissipate black energy.
\end{lem}
\begin{proof}
Because $t_1=t_2=t_3=\cdots =t_\ell <t_{\ell+1}\le t_{\ell +2} \le \cdots \le u_{d^-_G(u_j)}$, the atomic particles at vertices $u_i$, $1\le i \le \ell$ simultaneously arrive first at vertex $u_j$ and they merge to collide perfectly with the atomic particles allocated to $u_j$. In terms of the rules of the energy propagation model all other atomic particles along the other in-arcs of $u_j$ dissipate black energy. 
\end{proof}

To illustrate the use of Lemma \ref{Lem-3.2}, it is now applied to star graphs, paths and cycles. First, an obvious corollary.

\begin{cor}
$b^\bullet(G)=0 \Leftrightarrow\mathfrak{E}(G) = 0$.
\end{cor}
Note that for an energy graph $G$ the portion of black energy resulting from atomic particles dissipating their mass-energy equivalence is given by $b^\bullet(G)\cdot mc^2$ joules.
\begin{prop}
For star graphs $S_{1,n}$, $n\ge 3$; for paths $P_n$, $n\ge 2$ and for cycles $C_n$, $n\ge 3$ we have
\begin{enumerate}\itemsep0mm
	\item[(i)] $b^\bullet(S_{1,n})=0$ and $\mathfrak{E}(S_{1,n}) = 0,\forall\, \varphi(S_{1,n})$,
	\item[(ii)] $b^\bullet(P_n)=0$ and $\mathfrak{E}(P_n)=0$.
	\item[(iii)] $b^\bullet(C_n) =
	\begin{cases}
	0,  &  n\ \text{is even},\\
	1,  &  n\ \text{is odd}, 
	\end{cases}$ \\
	\textrm{and}
	$\mathfrak{E}(C_n) =
	\begin{cases}
	0,  &  n\ \text{is even},\\
	\frac{1}{2}\xi + mc^2,  &  n\ \text{is odd}. 
	\end{cases}$
\end{enumerate}
\end{prop}
\begin{proof}
Part (i): Consider a star graph $S_{1,n},\ n\ge 3$ and first orientate all edges as out-arcs from the central vertex. Label the central vertex as $u_1$ and randomly vertex label the pendant vertices $u_i, 2\le i \le n+1$. Clearly, at $t=1$ a single atomic particle has transited an arc to reach a corresponding pendant sink vertex. All arcs were transited simultaneously and hence $u_i\sim\langle 1\rangle,\ 2\le i \le n+1$. Hence, $b^\bullet(S_{1,n}) = 0$. Hence, $\mathfrak{E}(S_{1,n}) = 0$. Now take the inverse orientation and the reasoning remains the same.

\vspace{0.25cm}

Finally, and without loss of generality, consider the first orientation of the star graph and inverse the orientation of any number $\ell < n$ arcs. Now the star digraph has $\ell$ source vertices. At $t=1$, $\ell$ atomic particles will merge at $u_1$ for a perfect collision to transfer $\ell\cdot \xi$ joules to the $n-\ell$ atomic points allocated to $u_1$. So $u_1\sim \langle \underbrace{1,1,1,\ldots,1}_{\ell-entries}\rangle$. The aforesaid kinetic energy will equally divide amongst $n-\ell$ atomic particles, each transited an arc to reach a corresponding pendant sink vertex, $u_j$. For each sink vertex we have, $v_j\sim \langle 2\rangle$.  Total energy is conserved hence, $\mathfrak{E}(S_{1,n}) = 0$ implying $b^\bullet(S_{1,n}) = 0$. The result holds for all $\varphi(S_{1,n}),\ n\ge3$.

\vspace{0.25cm}

Part (ii): Consider any path $P_n$, $ n\ge 2$ and label the vertices to have the path $u_1u_2u_3\ldots u_n$ and orientate with arcs $(u_i,u_j)$, $1\le i \le n-1$ and $j=i+1$. Clearly, on igniting propagation of energy and atomic particle mass, we have $u_i\sim \langle i-1\rangle$, $1 \le i\le n$. Since a single atomic particle arrives at each vertex $u_j$, $j\ne 1$ no black edges exist. Therefore, $b^\bullet(P_n) = 0\Rightarrow \mathfrak{E}(P_n) = 0$.

\vspace{0.25cm}

Part (iii)(a): Consider a cycle $C_n$, $n\ge 4$, $n$ even. Let the vertices of $C_n$ be located along the circumference of a circle with vertex $u_1$ at the exact-top, and $u_n$ at the exact-bottom of the circle. From $u_1$ label the vertices located anticlockwise consecutively, $u_2,u_3,u_4,\ldots,u_{\frac{n}{2}}$ and orientate all corresponding edges including $u_{\frac{n}{2}}u_n$, anticlockwise. Similarly, from $u_1$ label vertices located clockwise consecutively, $u_{\frac{n}{2}+1}, u_{\frac{n}{2}+2},\ldots,u_{n-1}$ and orientate all corresponding edges including $u_{n-1}u_n$, clockwise. Clearly, this construction of an energy cycle does not contradict generality. Hence, on igniting propagation of energy and atomic particle mass we have $u_i\sim \langle i-1\rangle$, $1 \le i\le \frac{n}{2}$ and $u_{\frac{n}{2}+i} \sim\langle i\rangle$, $1\le i\le \frac{n}{2}-1$. Therefore,  $u_n\sim \langle \frac{n}{2},\frac{n}{2}\rangle$. In view of Lemma \ref{Lem-3.2}, no black arc exists. Therefore, $b^\bullet(C_n) = 0\Rightarrow \mathfrak{E}(C_n) = 0$ if $n$ even.

\vspace{0.25cm}

Part (iii)(b): The vertex labeling and orientation of a cycle $C_n$, where $n\ge 3$ and is odd, follow similar to Part iii(a) with the anticlockwise vertex labels running through $u_2,u_3,\ldots,u_{\lfloor\frac{n}{2}\rfloor}$ and the clockwise vertex labels running through $u_{\lfloor\frac{n}{2}\rfloor +1}, u_{\lfloor\frac{n}{2}\rfloor +2},\ldots, u_{n-1}$. Note that the arc $(u_{n-1},u_n)$ results in $u_n\sim \langle \lfloor \frac{n}{2}\rfloor, \lceil\frac{n}{2}\rceil\rangle$. So exactly one arc results in a black arc with black energy, $(\frac{1}{2}\xi + mc^2)$ joules. Therefore,  $b^\bullet(C_n) = 1 \Rightarrow \mathfrak{E} = \frac{1}{2}\xi + mc^2$ because the initial amount of potential energy capacitated at $u_1$ is $\xi$ joules.  
\end{proof}

\begin{thm}
An acyclic graph $G$ of order $n\ge 2$ has $b^\bullet(G) = 0$ and therefore, $\mathfrak{E}(G) = 0$.
\end{thm}
\begin{proof}
An acyclic graph $G$ is a bipartite graph. Partition $V(G)$ into disjoint subsets $X,Y$ and without loss of generality, orientate each edge $uw$, $u\in X$, $w \in Y$ to be the arc $(u,w)$. Label the vertices in $X$ as $u_1,u_2,u_3,\ldots,u_{|X|}$ and those in $Y$ as $w_{|X|+1}, w_{|X|+2},w_{|X|+3},\ldots,w_{|X|+|Y|}$. Clearly, the resultant vertex labeled digraph $G$ is an energy graph. Furthermore, at $t=1$ all atomic particles transit along an arc from a corresponding source vertex in $X$ to a sink vertex in $Y$. Therefore,  $b^\bullet(G) = 0$. It immediately follows that $\mathfrak{E}(G) = 0$.
\end{proof} 

\begin{cor}
A simple connected graph $G$ of order $n\ge 2$ which has no odd cycle has $b^\bullet(G) = 0$ and therefore, $\mathfrak{E}(G) = 0$. 
\end{cor}
\begin{proof}
It is well known that a graph which has no odd cycle is bipartite. The result then follows from the proof of Theorem 2.4. 
\end{proof}
An energy graph $G$ with $\mathfrak{E}(G) = 0$, ($b^\bullet(G) = 0$) is said to satisfy the \textit{law of conservation of total energy}. Characterisation of these graphs follows in the next result.
\begin{thm}
An energy graph $G$ of order $n\ge 2$ satisfies the law of conservation of total energy if and only if $G$ has no odd cycle.
\end{thm}
\begin{proof}
Case 1: If $G$ has no odd cycle the result follows from Corollary 2.6.

\vspace{0.25cm}

Case 2: Assume $G$ satisfies the law of conservation of total energy and has an odd cycle. Hence, $\mathfrak{E}(G) = 0$. However, since $m>0$ and $\xi > 0$ and order $n \in \N$ is finite, a vertex of the odd cycle receives or propagates some energy $> 0$ joules from in-arcs or along out-arcs, respectively. From Proposition 2.4(iii) it follows that for all arcs along the odd cycle, at least one such arc $(u_i,u_j)$ is a black arc in respect of the odd cycle as an induced subgraph of $G$. Thus, at least $(\frac{1}{2}mc^2 +\epsilon)$ joules dissipate as black energy along the arc $(u_i,u_j)$ during exhaustive propagation and $b^\bullet(G) > 1$. The aforesaid is a contradiction to the assumption. Therefore,  the result holds.
\end{proof}

The Jaco-type graphs $J_n(s_1)$, $J_n(s_2)$ clearly have well-defined black arcs and therefore it is possible to determine the number of black arcs. The black arc number can be determined through the \textit{Jaco-type Black Arc Algorithm} which is presented next.

\begin{defn}{\rm  In a finite Jaco-type graph $G$ of order $n\ge 2$ and a vertex $u_j \in V(G)$ the out-open neighbourhood is the set of head vertices of $u_j$, denoted $\mathcal{H}(u_j)$. The set of vertices of the induced subgraph $\langle \mathcal{H}(u_j)\rangle$ is called the black cloud of $u_j$. 
}\end{defn}

\subsection{Jaco-type Black Arc Algorithm}

Note that the algorithm is dense and requires iterative refinement for ICT application.

\ni \textit{Step 1}:- For a finite Jaco-type graph $G$, set $i= 1$, $\cB_0(G)= \emptyset$ and let $G_i = G$. Go to Step 2.

\vspace{0.25cm}

\ni \textit{Step 2}:- Set $j=i$ and consider vertex $u_j \in V(G_j)$ and determine $\langle \mathcal{H}(u_j)\rangle$. Go to Step 3.

\vspace{0.25cm}

\ni \textit{Step 3}:- Let $\cB_j(G) =\cB_{j-1}(G) \cup E(\langle\mathcal{H}(u_j)\rangle)$. If $j= n-1$, go to Step 4. Else, let $i=j+1$ and $G_i = G_j - E(\langle\mathcal{H}(u_j)\rangle)$. Go to Step 2.

\vspace{0.25cm}

\ni \textit{Step 4}:- Set $b^\bullet(G) = |\cB_{n-1}(G)|$ and exit.
  
\vspace{0.25cm}
 
We call the set $\cB_{n-1}(G)$ the \textit{black cloud} of the graph $G$. Note that $b^\bullet(G) = |\cB_{n-1}(G)| \le |\bigcup\limits_{j=1}^{n}E(\langle \mathcal{H}(u_j)\rangle)|$. 

\begin{thm}
For a finite Jaco-type graph $G$ the Jaco-type Black Arc Algorithm is well-defined and it converges.
\end{thm}
\begin{proof}
Clearly, Step 1 is unambiguous and finite. Clearly, Step 2 is unambiguous and since $G$ is finite the range $j \le n-1$ in Step 3 is well-defined and finite which implies Step 2 converges. Furthermore, since the out-neighbourhood of any vertex is well-defined and arcs are unambiguous in a simple digraph, the induced subgraph $\langle \mathcal{H}(u_j)\rangle$ is well-defined. Since $G_j$ is finite, determining $\langle \mathcal{H}(u_j)\rangle$ converges. Therefore,  Step 3 converges. Step 4 is unambiguous and finite. Hence, the Jaco-type Black Arc Algorithm is well-defined and it converges.
\end{proof}

\begin{thm}
The Jaco-type Black Arc Algorithm determines all the black arcs of a finite Jaco-type graph $G$ of order $n\ge 2$.
\end{thm}
\begin{proof}
At any time $t_i$ which corresponds with the $i^{th}$-iteration the vertex $u_i \in V(G_i)$ is under consideration. If any two out-neighbours of $u_i$ say vertices $u_k,u_\ell$, $k<\ell$ are adjacent the arc $(u_k,u_\ell)$ exists. No energy transmission will be possible from $u_k$ to $u_\ell$ at time $t_{i+1}$ hence, all energy allocated at vertex $u_k$ will dissipate into black energy at $t_{i+1}$. Therefore,  the Jaco-type Black Arc Algorithm determines all black arcs of $G$. 
\end{proof}

\begin{app}{\rm 
Applying the Jaco-type Black Arc Algorithm on $J_8(s_1)$ and $J_{12}(s_2)$, we have  

\begin{eqnarray*}
\cB_7(J_8(s_1)) & = & \{(u_3,u_4),(u_5,u_6),(u_5,u_7),(u_5,u_8),(u_6,u_7),(u_6,u_8),(u_7,u_8)\}, \\
\cB_{11}(J_{12}(s_2)) & = & \{(u_4,u_5),(u_6,u_7),(u_6,u_8),(u_6,u_9),(u_6,u_{10}),(u_7,u_8),(u_7,u_9),\\ & & (u_7,u_{10}), (u_8,u_9),(u_8,u_{10}),(u_9,u_{10}), (u_{11},u_{12})\}
\end{eqnarray*}

\ni Therefore,  $b^\bullet(J_8(s_1)) = 7$, $b^\bullet(J_{12}(s_2)) = 12$.}
\end{app}

The notion of primitive holes in graphs has been introduced in \cite{KS1} and defined the primitive degree, $d^p_G(u)$ of a vertex $u\in V(G)$. Let the graphs $G_i$, $1\le i\le n-1$ be those resulting from the Jaco-type Black Arc Algorithm. Denote the underlying graph of $G$ by $G^*$. The next theorem presents $b^\bullet(G)$ in terms of primitive degrees found in Jaco-type graphs.

\begin{thm}
For a finite Jaco-type graph $G$ of order $n\ge 2$ we have $b^\bullet(G) = \sum\limits_{i=1}^{n-1}h^p_{G^*_i}(u_i)$.
\end{thm}
\begin{proof}
Lemma 2.1 states that a Jaco-type graph is an energy graph. From definitions 2.1 and 2.2 it follow directly that each arc of the black cloud $\langle \mathcal{H}(u_i)\rangle$ is an edge of a primitive hole of $G^*_i$ with the common vertex $u_i$. So the number of black edges associated with $u_i \in V(G^*_i)$ equals $d^P_{G^*}(u_i)$. Hence, the result $b^\bullet(G) = \sum\limits_{i=1}^{n-1}h^p_{G^*_i}(u_i)$ is settled.
\end{proof}

The subgraph $G- \bigcup\limits_{i=1}^{n-1}E(\langle \mathcal{H}(u_i)\rangle) = G-\cB_{n-1}(G)$ is called the \textit{solid subgraph} of the propagating graph $G$, denoted $G^s$. For a given energy graph $G$ (given orientation), the corresponding solid subgraph is unique. We now have a useful lemma, with trivial proof.

\begin{lem}
The number of arcs, $|\cB_{n-1}(G)|$ equals the number of atomic particles which dissipated into mass-energy equivalence when propagation exhausted. Similarly the number of arcs, $|A(G^s)|$ equals the number of atomic particles which remained within the energy graph when propagation exhausted.
\end{lem}

\begin{app}{\rm 
For the Jaco-type graph $J_8(s_1)$, the number of atomic particles remained is $|A(J_8(s_1))|-b^\bullet(J_8(s_1))=16-7=9=|A(J^s_8(s_1))|$ only and the number of atomic particles dissipated is $7$. In $J_{12}(s_2)$,  only $|A(J_{12}(s_2))|- b^\bullet(J_{12}(s_2))=33-12=21=|A(J^s_{12}(s_2))|$ atomic particles remained and $12$ dissipated.}
\end{app}

\section{Application to Binary Code, Gray Code and Modular Arithmetic Jaco-type Graphs}

For a binary code of bit width $n$, the the binary combinations are always finite and an even number of combinations (that is, $2^n$ numbers) exist. The binary combinations will be tabled such that consecutive rows maps to consecutive decimal representations $b_i = i-1$, where $ i = 1,2,3,\ldots,2^n$ of the binary code. 

\vspace{0.2cm}

Table 1 serves as an example of the convention for a binary code of \textit{bit width} $3$.

\begin{table}[h!]\label{Tab-1}
\begin{center}
\begin{tabular}{|c|c|c|c|}
\hline
0  &  0  &  0  &  $\mapsto 0= b_1$\\
\hline
0 & 0 & 1 & $\mapsto 1 = b_2$\\
\hline
0 &1 & 0 & $\mapsto 2 = b_3$\\
\hline
0 & 1 & 1 & $\mapsto 3 = b_4$\\
\hline
1 & 0 & 0 & $\mapsto 4 = b_5$\\
\hline
1 & 0 & 1 & $\mapsto 5 = b_6$\\
\hline
1 & 1 & 0 & $\mapsto 6 = b_7$\\
\hline
1 & 1 & 1 & $\mapsto 7 = b_8$\\
\hline
\end{tabular}
\caption{}
\end{center}
\end{table}

Contrary to the general finite Jaco-type graphs $J_n(\{a_n\}),\ n \in \N$ the graphical embodiment of the binary codes denoted, $\mathcal{G}_n$ for $n \in \N$ is defined finite on exactly $2^{n+1}$ vertices, $u_1, u_2, u_3,\ldots,u_{2^n},u_{2^n +1},u_{2^n + 2},\ldots, u_{2^{n+1}}$. We define $d^+(u_i) = b_i + 1$, $1\le i \le 2^n$ and $d^+(u_j) = 2^n-k$, $j = 2^n + k$, $k= 1,2,3,\ldots,2^n$.  

It follows easily that $\mathcal{G}_n \simeq J_{2^{n+1}}(s_1)$. See Figure \ref{fig:Fig-1} for $\mathcal{G}_2 \simeq J_8(s_1)$. The disconnected graph obtained from the union of $t \ge 2$ copies of $\mathcal{G}_n$ has vertices labeled such that the corresponding vertices in the $i^{th}$-copy, $i \ge 2$ are $ u_{(i-1)2^{n+1} +j}$, $2\le i \le t$, $1\le j \le 2^{n+1}$. Through easy sequential counting and indexing known results for the finite Jaco-type graph $J_{2^n+1}(s_1)$ can be extrapolated. For example, for $\bigcup\limits_{t-copies}\mathcal{G}_n$, $\mathcal{G}_n$ the graphical embodiment of the binary code with bit width $n$ we have $\J(\bigcup\limits_{t-copies}\mathcal{G}_n) = \{u_{p+j}: u_p \in \J(J_{2^{n+1}}(s_1)), j=0,1,2,\ldots,(t-1)2^{n+1}\}$; $|\J(\bigcup\limits_{t-copies}\mathcal{G}_n)| = t\cdot |\J(J_{2^{n+1}}(s_1))|$ and $b^\bullet(\bigcup\limits_{t-copies}\mathcal{G}_n) = t\cdot b^\bullet(J_{2^{n+1}}(s_1)$.

\vspace{0.2cm}

The analysis for Gray codes follows similarly to that for binary codes since the Gray combinations can be mapped onto exactly the same decimal representations.

\subsection{Jaco-type Graph for Sequences Modulo $k$}

It is well known that for $\N_0$ and $n,k\in \N$, $k\ge 2$ modular arithmetic allows an integer mapping in respect of mod $k$ as follows: $0\mapsto 0 =m_0$, $ 1\mapsto 1= m_1$, $ 2 \mapsto 2 =m_2,\ldots, k-1 \mapsto k-1 = m_{k-1}$, $k\mapsto 0 = m_k$, $k+1 \mapsto 1 =m_{k+1}, \ldots$ Let $s_3 =\{a_n\}$, $ a_n \equiv n($mod $k)= m_n$. Consider the infinite \textit{root}-graph $J_\infty(s_3)$ and define $d^+(u_i) = m_i$, for $i = 1,2,3,\ldots$.

From the aforesaid definition it follows that the case $k=1$ will result in a null (edgeless) Jaco-type graph $\forall\, n \in \N$. For the case $k=2$ the Jaco-type graph for $n\ge 2$ and even, is the union of $\frac{n}{2}$ copies of directed $P_2$. For the case $k=3$ the Jaco-type graph is a directed tree hence, an acyclic graph $G$ and therefore $b^\bullet(G)=0$. The smallest mod $k$ Jaco-type graph which has an black arc is $J_4(s_3)$, $k=4$. 

\vspace{0.2cm}

\ni Figure \ref{fig:Fig-3} depicts $J_{12}(s_3)$ for $k=5$.

\begin{figure}[h!]
\centering
\includegraphics[width=0.75\linewidth]{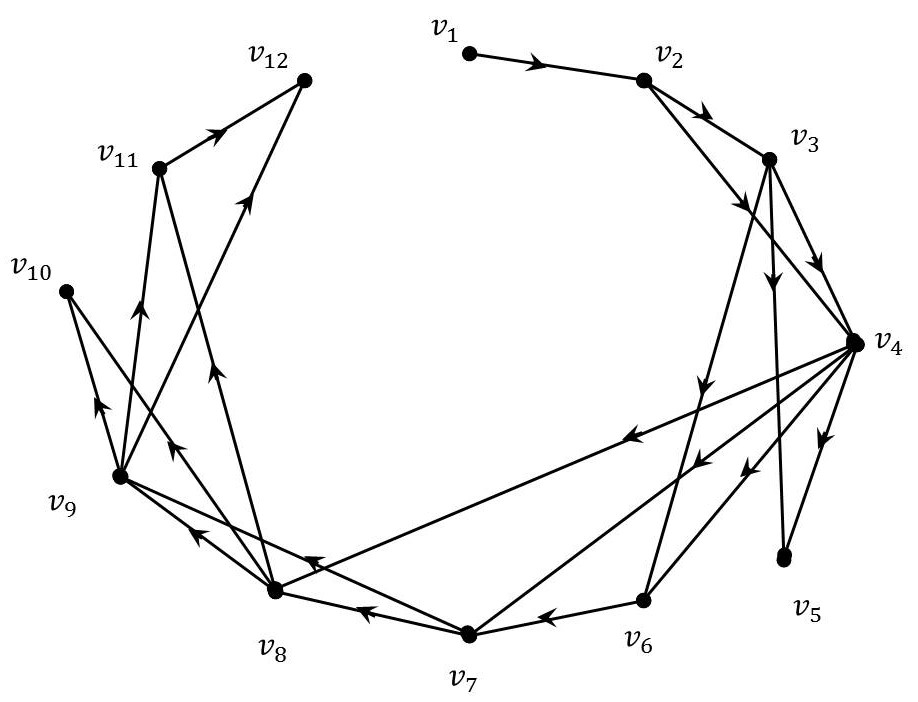}
\caption{$J_{12}(s_3),\ k = 5$}
\label{fig:Fig-3}
\end{figure}

We note that a Jaco-type graph $J_n(s_3)$ has $\lceil\frac{n}{k}\rceil$ sink vertices. The sink vertices resulting from the modular function is called \textit{mod-sink vertices}. It follows from the Jaco-type Black Arc Algorithm that $b^\bullet(J_{12}(s_3)) = 5$. A total of $(\frac{7}{24}\xi + 2mc^2)$ joules, $(\frac{1}{24}\xi + mc^2)$ joules and $(\frac{1}{12}\xi + 2mc^2)$ joules reach the sink vertices $v_5, v_{10}, v_{12}$ respectively. A total of $(\frac{47}{72}\xi + 6mc^2)$ joules dissipate as black energy along black arcs. Therefore, a total of $(\frac{1}{12}\xi + 14mc^2)$ joules are capacitated within the graph.
\begin{lem}
A mod-sink vertex $u_j$ of the Jaco-type graphs $J_n(s_3)$, $k\ge 3$ has $d^-(u_j) = \lfloor\frac{k}{2}\rfloor$.
\end{lem}
\begin{proof}
Consider any mod-sink vertex $u_j$ of the Jaco-type graph $J_n(s_3)$, $n \ge k \ge 4$. Clearly, the minimum $t$ for which the arc $(u_t,u_j)$ exists is $\lfloor\frac{n}{2}\rfloor$. Hence, the result. 
\end{proof}
\textbf{For the case $k=4$, $4 \le n\le 35$:} Applying the Jaco-type Black Arc Algorithm iteratively leads to the next result.

\begin{prop}
For $s_3 = \{a_n\}$, $a_n \equiv n($mod $4)$ the Jaco-type graph $J_n(s_3)$ has:\\
(i)  $b^\bullet(J_n(s_3)) = \lfloor\frac{n}{2}\rfloor - 1$, alternatively;\\
(ii)\begin{equation*} 
b^\bullet(J_n(s_3)) =
\begin{cases}
i,  & \text {$\forall\, n= 3 +2i$, $i=0,1,2,\ldots$ } and,\\
i + 1,  &  \text {$\forall\, n > max\{\ell:\ell < n$ and $\ell = 3 + 2i$, $i=0,1,2,\ldots\}$}. 
\end{cases}
\end{equation*} 
\end{prop}
\begin{proof}
\textit{Part (i):} It is trivially true that $b^\bullet(J_n(s_3)) \le b^\bullet(J_{n+1}(s_3))$. Applying the Jaco-type Black Arc Algorithm to both $J_4(s_3),J_5(s_3)$ results in the only black arc, $(u_3,u_4)$ for both thus, $b^\bullet(J_4(s_3)) = b^\bullet(J_5(s_3))=1 =\lfloor\frac{4}{2}\rfloor -1 = \lfloor\frac{5}{2}\rfloor -1$. Thus the result holds for the integer pair, $n = 4,5$. Assume the results holds for the integer pair $n=t, t+1$, $t\in \N$. Hence, $b^\bullet(J_t(s_3)) =\lfloor\frac{t}{2}\rfloor -1 = \lfloor\frac{t+1}{2}\rfloor -1 = b^\bullet(J_{t+1}(s_3))$.

\vspace{0.2cm}

Observe that if $t$ is even then $\lfloor\frac{t}{2}\rfloor -1 = \lfloor \frac{t+1}{2}\rfloor -1$. Since for, $t'' = t +2$ which could have been the initial even integer $\in \N$ by the induction assumption, we have that $\lfloor \frac{t''}{2}\rfloor -1 = (\lfloor \frac{t}{2}\rfloor +1) -1 \implies b^\bullet(J_{t''}(s_3)) = b^\bullet(J_t(s_3)) + 1 = b^\bullet(J_{t+1}(s_3)) + 1= \lfloor\frac{t''+1}{2}\rfloor\rfloor -1 = b^\bullet(J_{t''+1}(s_3))$. Hence the result holds for the integer pair $t''$ and $t''+1$ or put another way, it holds for the pair of Jaco-type graphs, $J_{t''}(s_3),J_{t''+1}(s_3)$ that $b^\bullet(J_{t''}(s_3))  = \lfloor \frac {t''}{2}\rfloor -1 = \lfloor \frac {t''+1}{2}\rfloor -1 = b^\bullet(J_{t''+1}(s_3))$. 
The general result $\forall\, n\ge 4$, $ n\in \N$ follows by mathematical induction.

\vspace{0.2cm}

\textit{Part (ii):} Alternatively:

\vspace{0.2cm}

\textit{Part(ii)(a):} $\forall\, n = 3 + 2i$, $ i=0,1,2\ldots$ it follows from Case (i) that $\lfloor \frac{3+2i}{2}\rfloor -1 = \lfloor \frac{3}{2} + \frac{2i}{2}\rfloor-1=i$.

\vspace{0.2cm}

\textit{Part (ii)(b):} Since $3+2i$, $i=0,1,2\ldots$ is odd, the integer $n > max\{\ell:\ell < n$ and $\ell = 3 + 2i$, $i=0,1,2,\ldots\}\Rightarrow n= \ell +1$ and even. Therefore,  from Case (i) it follows that $b^\bullet(J_n(s_3)) = i + 1$.
\end{proof}

If $k=5$, applying the Jaco-type Black Arc Algorithm iteratively for $n\ge 4$, we get the following table.

\begin{table}[h!]
\begin{center}
\begin{tabular}{|c|c|c|c|c|c|c|c|c|c|c|c|}
\hline
$n$ & 4 & 5 & 6 & 7 & 8 & 9 & 10 & 11 & 12 & 13 & 14 \\ 
\hline
$b^\bullet(J_n(s_3))$ & 1 & 1 & 1 & 2 & 3 & 3 & 4 & 5 & 5 & 6 & 6\\
\hline
\hline
$n$ & 15 & 16 & 17 & 18 & 19 & 20 & 21 & 22 & 23 & 24 & 25 \\ 
\hline
$b^\bullet(J_n(s_3))$ & 7 & 7 & 8 & 8 & 9 & 10 & 11 & 11 & 12 & 12 & 13 \\ 
\hline
\hline
$n$ & 26 & 27 & 28 & 29  & 30 & 31 & 32 & 33 & 34 & 35 & \\ 
\hline
$b^\bullet(J_n(s_3))$ & 14 & 14 & 15 & 15 & 16 & 16 & 17 & 18 & 18 & 19 & \\ 
\hline
\end{tabular}
\caption{}
\end{center}
\end{table}

\ni The black clouds are 
\begin{eqnarray*}
\cB_3(J_4(s_3)) & = & \{(u_3,u_4)\},\\
\cB_4(J_5(s_3)) & = & \{(u_3,u_4)\}, \\
\cB_5(J_6(s_3)) & = & \{(u_3,u_4)\}, \\
\cB_6(J_7(s_3)) & = & \{(u_3,u_4),(u_6,u_7)\}, \\
\cB_7(J_8(s_3)) & = & \{(u_3,u_4),(u_6,u_7),(u_7,u_8)\},\\ 
\cB_8(J_9(s_3)) & = & \{(u_3,u_4),(u_6,u_7),(u_7,u_8)\},\\ 
\cB_9(J_{10}(s_3)) & = & \{(u_3,u_4),(u_6,u_7),(u_7,u_8),(u_9,u_{10})\},\\ 
\cB_{10}(J_{11}(s_3)) & = & \{(u_3,u_4),(u_6,u_7),(u_7,u_8),(u_9,u_{10}),(u_9,u_{11})\},\\ \cB_{11}(J_{12}(s_3)) & = & \{(u_3,u_4),(u_6,u_7),(u_7,u_8),(u_9,u_{10}),(u_9,u_{11})\},\\ \cB_{12}(J_{13}(s_3)) & = & \{(u_3,u_4),(u_6,u_7),(u_7,u_8),(u_9,u_{10}),(u_9,u_{11}),(u_{12},u_{13})\},\\ \cB_{13}(J_{14}(s_3)) & = & \{(u_3,u_4),(u_6,u_7),(u_7,u_8),(u_9,u_{10}),(u_9,u_{11}),(u_{12},u_{13})\},\\ \cB_{14}(J_{15}(s_3)) & = & \{(u_3,u_4),(u_6,u_7),(u_7,u_8),(u_9,u_{10}),(u_9,u_{11}),(u_{12},u_{13}),(u_{14},u_{15})\},\\ \cB_{15}(J_{16}(s_3)) & = & \{(u_3,u_4),(u_6,u_7),(u_7,u_8),(u_9,u_{10}),(u_9,u_{11}),(u_{12},u_{13}),(u_{14},u_{15}),\\ & & 
(u_{14},u_{16})\},\\ 
\cB_{16}(J_{17}(s_3)) & = &  \{(u_3,u_4),(u_6,u_7),(u_7,u_8),(u_9,u_{10}),(u_9,u_{11}),(u_{12},u_{13}),(u_{14},u_{15}),\\ & & 
(u_{14},u_{16})\},\\ 
\cB_{17}(J_{18}(s_3)) & = &  \{(u_3,u_4),(u_6,u_7),(u_7,u_8),(u_9,u_{10}),(u_9,u_{11}),(u_{12},u_{13}),(u_{14},u_{15}),\\ & &  
(u_{14},u_{16}),(u_{17},u_{18})\},\\ 
\cB_{18}(J_{19}(s_3)) & = & \{(u_3,u_4),(u_6,u_7),(u_7,u_8),(u_9,u_{10}),(u_9,u_{11}),(u_{12},u_{13}),(u_{14},u_{15}),\\ & &  
(u_{14},u_{16}),(u_{17},u_{18})\},\\
\cB_{19}(J_{20}(s_3)) & = & \{(u_3,u_4),(u_6,u_7),(u_7,u_8),(u_9,u_{10}),(u_9,u_{11}),(u_{12},u_{13}),(u_{14},u_{15}),\\ & & 
(u_{14},u_{16}),(u_{17},u_{18}),(u_{19},u_{20})\},\\  
\cB_{20}(J_{21}(s_3)) & = & \{(u_3,u_4),(u_6,u_7),(u_7,u_8),(u_9,u_{10}),(u_9,u_{11}),(u_{12},u_{13}),(u_{14},u_{15}), \\ & & (u_{14},u_{16}),(u_{17},u_{18}),(u_{19},u_{20}),(u_{19},u_{21})\},\\ 
\cB_{21}(J_{22}(s_3)) & = & \{(u_3,u_4),(u_6,u_7),(u_7,u_8),(u_9,u_{10}),(u_9,u_{11}),(u_{12},u_{13}),(u_{14},u_{15}),\\ & & (u_{14},u_{16}),(u_{17},u_{18}),(u_{19},u_{20}),(u_{19},u_{21})\},\\ 
\cB_{22}(J_{23}(s_3)) & = & \{(u_3,u_4),(u_6,u_7),(u_7,u_8),(u_9,u_{10}),(u_9,u_{11}),(u_{12},u_{13}),(u_{14},u_{15}), \\ & & (u_{14},u_{16}),(u_{17},u_{18}),(u_{19},u_{20}),(u_{19},u_{21}),(u_{22},u_{23})\},\\ 
\cB_{23}(J_{24}(s_3)) & = & \{(u_3,u_4),(u_6,u_7),(u_7,u_8),(u_9,u_{10}),(u_9,u_{11}),(u_{12},u_{13}),(u_{14},u_{15}),\\ & & (u_{14},u_{16}),(u_{17},u_{18}),(u_{19},u_{20}),(u_{19},u_{21}),(u_{22},u_{23})\},\\ 
\cB_{24}(J_{25}(s_3)) & = & \{(u_3,u_4),(u_6,u_7),(u_7,u_8),(u_9,u_{10}),(u_9,u_{11}),(u_{12},u_{13}),(u_{14},u_{15}),\\ & &  (u_{14},u_{16}),(u_{17},u_{18}),(u_{19},u_{20}),(u_{19},u_{21}),(u_{22},u_{23}),(u_{24},u_{25})\},\\ 
\cB_{25}(J_{26}(s_3)) & = & \{(u_3,u_4),(u_6,u_7),(u_7,u_8),(u_9,u_{10}),(u_9,u_{11}),(u_{12},u_{13}),(u_{14},u_{15}),\\ & & (u_{14},u_{16}),(u_{17},u_{18}), (u_{19},u_{20}),(u_{19},u_{21}),(u_{22},u_{23}),(u_{24},u_{25}),\\ & & (u_{24},u_{26})\},\\
\cB_{26}(J_{27}(s_3)) & = & \{(u_3,u_4),(u_6,u_7),(u_7,u_8),(u_9,u_{10}),(u_9,u_{11}),(u_{12},u_{13}),(u_{14},u_{15}),\\ & & (u_{14},u_{16}),(u_{17},u_{18}), (u_{19},u_{20}),(u_{19},u_{21}),(u_{22},u_{23}),(u_{24},u_{25}),\\ & & (u_{24},u_{26})\},\\  
\cB_{27}(J_{28}(s_3)) & = & \{(u_3,u_4),(u_6,u_7),(u_7,u_8),(u_9,u_{10}),(u_9,u_{11}),(u_{12},u_{13}),(u_{14},u_{15}),\\ & & (u_{14},u_{16}),(u_{17},u_{18}), (u_{19},u_{20}),(u_{19},u_{21}),(u_{22},u_{23}),(u_{24},u_{25}),\\  & &  (u_{24},u_{26}),
(u_{27},u_{28})\},\\ 
\cB_{28}(J_{29}(s_3)) & = & \{(u_3,u_4),(u_6,u_7),(u_7,u_8),(u_9,u_{10}),(u_9,u_{11}),(u_{12},u_{13}),(u_{14},u_{15}),\\ & & (u_{14},u_{16}),(u_{17},u_{18}), (u_{19},u_{20}),(u_{19},u_{21}),(u_{22},u_{23}),(u_{24},u_{25}),\\  & &  (u_{24},u_{26}), (u_{27},u_{28})\}\\ 
\cB_{29}(J_{30}(s_3)) & = & \{(u_3,u_4),(u_6,u_7),(u_7,u_8),(u_9,u_{10}),(u_9,u_{11}),(u_{12},u_{13}),(u_{14},u_{15}),\\ & & (u_{14},u_{16}),(u_{17},u_{18}), (u_{19},u_{20}),(u_{19},u_{21}),(u_{22},u_{23}),(u_{24},u_{25}),\\  & &  (u_{24},u_{26}), (u_{27},u_{28}),(u_{29},u_{30})\},\\ 
\cB_{30}(J_{31}(s_3)) & = & \{(u_3,u_4),(u_6,u_7),(u_7,u_8),(u_9,u_{10}),(u_9,u_{11}),(u_{12},u_{13}),(u_{14},u_{15}),\\ & & (u_{14},u_{16}),(u_{17},u_{18}), (u_{19},u_{20}),(u_{19},u_{21}),(u_{22},u_{23}),(u_{24},u_{25}), \\ & & (u_{24},u_{26}), (u_{27},u_{28}),(u_{29},u_{30})\},\\ 
\cB_{31}(J_{32}(s_3)) & = & \{(u_3,u_4),(u_6,u_7),(u_7,u_8),(u_9,u_{10}),(u_9,u_{11}),(u_{12},u_{13}),(u_{14},u_{15}), \\ & & (u_{14},u_{16}),(u_{17},u_{18}), (u_{19},u_{20}),(u_{19},u_{21}),(u_{22},u_{23}),(u_{24},u_{25}),\\  & &  (u_{24},u_{26}), (u_{27},u_{28}),(u_{29},u_{30}),(u_{31},u_{32})\},\\
\cB_{32}(J_{33}(s_3)) & = & \{(u_3,u_4),(u_6,u_7),(u_7,u_8),(u_9,u_{10}),(u_9,u_{11}),(u_{12},u_{13}),(u_{14},u_{15}),\\ & & (u_{14},u_{16}),(u_{17},u_{18}), (u_{19},u_{20}),(u_{19},u_{21}),(u_{22},u_{23}),(u_{24},u_{25}), \\  & &  (u_{24},u_{26}), (u_{27},u_{28}),(u_{29},u_{30}),(u_{31},u_{32}),(u_{32},u_{33})\},\\
\cB_{33}(J_{34}(s_3)) & = & \{(u_3,u_4),(u_6,u_7),(u_7,u_8),(u_9,u_{10}),(u_9,u_{11}),(u_{12},u_{13}),(u_{14},u_{15}),\\ & &  (u_{14},u_{16}),(u_{17},u_{18}), (u_{19},u_{20}),(u_{19},u_{21}),(u_{22},u_{23}),(u_{24},u_{25}),\\ & & (u_{24},u_{26}), (u_{27},u_{28}),(u_{29},u_{30}),(u_{31},u_{32}),(u_{32},u_{33})\},\\ 
\cB_{34}(J_{35}(s_3)) & = & \{(u_3,u_4),(u_6,u_7),(u_7,u_8),(u_9,u_{10}),(u_9,u_{11}),(u_{12},u_{13}),(u_{14},u_{15}),\\ & & (u_{14},u_{16}),(u_{17},u_{18}), (u_{19},u_{20}),(u_{19},u_{21}),(u_{22},u_{23}),(u_{24},u_{25}),\\ & & (u_{24},u_{26}),  (u_{27},u_{28}),(u_{29},u_{30}),(u_{31},u_{32}),(u_{32},u_{33}),(u_{34},u_{35})\}.
\end{eqnarray*}

Figure \ref{fig:Fig-4} depicts $J_{18}(s_3)$.

\begin{figure}[h!]
\centering
\includegraphics[width=0.8\linewidth]{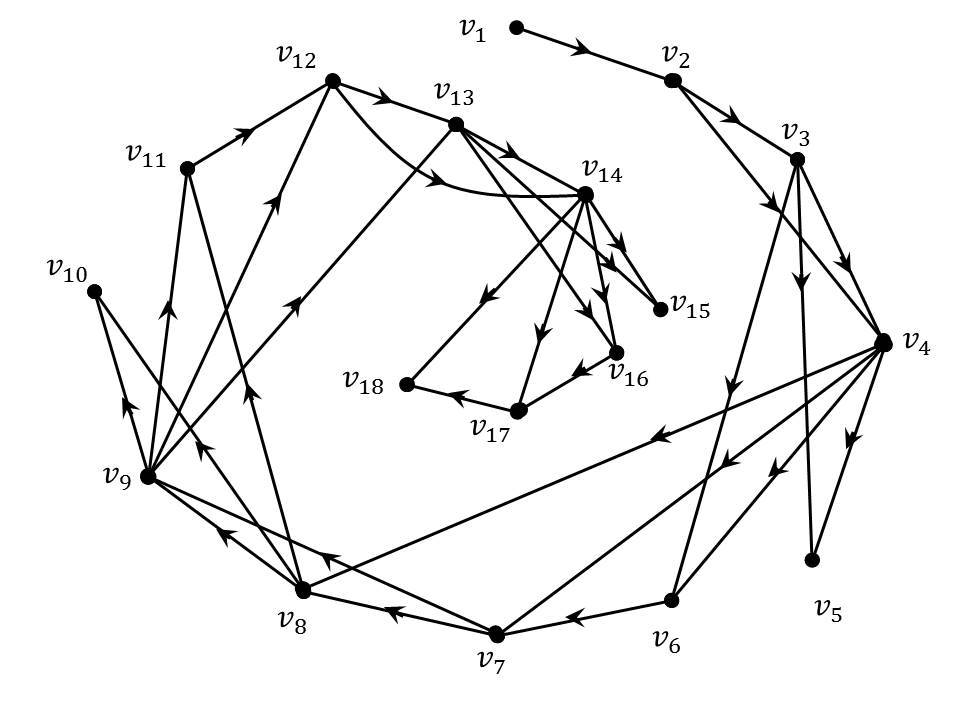}
\caption{$J_{18}(s_3)$.}
\label{fig:Fig-4}
\end{figure}

\section{Application of a General Black Arc Algorithm}

A general Black Arc Algorithm is not known yet. However, in the event of finding such an algorithm, the complexity of application to energy graphs with cut vertices can be simplified through partitioned application. It is known that a vertex $u_i \in V(G)$ is a cut vertex of $G$ if and only if the set of arcs $A(G)$ can be partitioned into subsets $A_1(G)$, $A_2(G)$ and the induced arc-subgraphs $\langle A_1(G)\rangle$ and $\langle A_2(G)\rangle$ have only vertex $u_i$ in common.

\begin{thm}
Let the energy graph $G$ have $c$ cut vertices and the subsets $A_1(G),\\ A_2(G), A_3(G),\ldots, A_\ell(G)$ of $A(G)$ form the maximal partition of $A(G)$ such that all pairs of distinct induced arc-subgraphs $\langle A_i(G)\rangle$ and $\langle A_j(G)\rangle$, where $i,j \in \{1,2,\ldots,\ell\}$, have at most one vertex in common. Then, $b^\bullet(G) = \sum\limits_{i=1}^{\ell} b^\bullet(\langle A_i(G)\rangle)$.
\end{thm}
\begin{proof}
Clearly, an induced arc-subgraph $\langle A_i(G)\rangle$ has no cut vertex else, the arc partitioning is not a maximum. Furthermore,   for those distinct pairs of induced arc-subgraphs which have a vertex  $u$ in common, such $u$ is a cut vertex of $G$. It is observed that for any induced arc-subgraph $\langle A_i(G)\rangle$ at least one induced arc-subgraph $\langle A_j(G)\rangle$ exists such that a common vertex $u \in V(G)$ between the arc-induced subgraphs exists else, $G$ is disconnected. Consider any distinct pair of arc-induced subgraphs $\langle A_i(G)\rangle$ and $\langle A_j(G)\rangle$ which share the common vertex $u$.

\vspace{0.2cm}

\textit{Case 1:} If $u$ is a source vertex in $G$ then $u$ is a source vertex in $\langle A_i(G)\rangle$ and $\langle A_j(G)\rangle$, respectively. Hence during the time interval $[0,1)$ of propagation no arc is rendered a black arc.

\vspace{0.2cm}

\textit{Case 2:} If $u$ is a sink vertex in $G$ then $u$ is a sink vertex in $\langle A_i(G)\rangle$ and $\langle A_j(G)\rangle$, respectively. Let $t',t''$ be the total (exhaustive) propagation time through $\langle A_i(G)\rangle$ and $\langle A_j(G)\rangle$, respectively.  Since a sink vertex can never initiate further propagation it cannot propagate kinetic energy or particle mass which could dissipate. Hence, during the time interval $[0,t')$ of propagation the arcs in $\langle A_i(G)\rangle$ rendered black arcs are identical to those found in $G$.  Similar reasoning applies to $\langle A_j(G)\rangle$ over the propagation period $[0,t'')$.

\vspace{0.2cm}

\textit{Case 3:} If $u$ is an intermediate ($d_G^+(u) > 0$, $d^-_G(u) >0$) vertex the independence of both propagation and the creation of black arcs lie in the fact that for $u_l \in V(\langle A_i(G)\rangle)$ and $u_k \in V(\langle A_j(G)\rangle)$ the arc $(u_l,u_k) \notin A(G)$. Therefore, the black arcs resulting from propagation through $\langle A_i(G)\rangle$ during $[0,t')$ and those resulting from propagation through $\langle A_j(G)\rangle$ during $[0,t'')$ are independent in each induced arc-subgraph and therefore identical to those found in $G$. 

\vspace{0.2cm}

Therefore,  $b^\bullet(G) = \sum\limits_{i=1}^{\ell} b^\bullet(\langle A_i(G)\rangle)$, completing the proof.
\end{proof}

\section{Conclusion}

There is a wide scope for further research in respect of total dissipated black energy and determining the black arc number for different classes of graphs. The Jaco-type Black Arc Algorithm poses the challenge of complexity analysis. Such analysis will contribute to theoretical computer science. Determining the amount of kinetic energy which dissipate from an energy graph $G$ is an open problem. Studying the properties of solid subgraphs is also wide open. Some of the other open problems we have identified during our study are the following.

\begin{prob}{\rm 
Verify whether the solid subgraph of an energy graph $G$ is connected.}
\end{prob}

\begin{prob}{\rm 
Find a closed formula for $b^\bullet(J_n(s_3))$ where $s_3 = \{$mod 5 sequence$\}$, either in terms of the black arc number or in terms of appropriate vertex indices.}
\end{prob}

\begin{prob}{\rm 
Describe a Black Arc Algorithm for a general energy graph.}
\end{prob}

\begin{rem}{\rm 
A Jaco-type graph can alternatively be defined as follows. A directed graph $G$ is a Jaco-type graph if the vertex set is the set of positive integers and there is a non-decreasing sequence $s=\{a_n\}_{n\in \N}$ of positive integers such that a pair $(n,m)$ is an arc of $G$ if and only if $n<m\leq a_n +n$. }
\end{rem}

An interesting, directed complement graph exists namely, the graph $G'$ with arcs $(n,m)$ such that $n<m$ and $(n,m)\notin A(G)$. Hence, all pairs $(n,m)$ such that $a_n+n<m$ belong to $G'$ so a strict order on $\N$ is defined. Ordered sets obtained in this way have a \textit{minimal type}. It means the sets are infinite and no proper initial segment is finite.

\vspace{0.2cm}

Conversely, a poset $P$ has minimal type if and only if elements can be enumerated in a sequence$\{v_n\}_{n \in \N}$ and there is a sequence $\{a_n\}_{n\in \N}$ of positive integers such that $v_n<v_m$ if $n+a_n < m$. Also an \textit{interval order} is an ordered st $P$ in which the poset $Q$ made of the direct sum of two 2-element chains does not embed. If moreover, $P$ does not embed the poset made of the direct sum of a 3-element chain and a 1-element chain then $P$ is a semiorder. Also, for interval orders both the predecessor and the successor order are total quasi-order. It implies that a graph $G$ is a Jaco-type graph with enumeration $\{v_n\}_n$ if and only if the order attached to the enumeration is the predecessor order that corresponds to the order $G'$ on the directed complement as defined above. Finally, an order is minimal if and only if it extends the directed complement of a Jaco-type graph.

\vspace{0.2cm}

A valid observation is that the sequence of positive integers need not necessarily be non-decreasing. This observation and the perspective from a set theory point of view opens further research. 

\section*{Acknowledgements}

The authors of this article gratefully acknowledge the critical and constructive comments of the anonymous referee, which significantly improved the content and presentation of this article.

\end{document}